\documentclass[suppldata]{interact}
 \usepackage{hyperref}
\usepackage{amsmath,amsthm}
\usepackage{epstopdf}
\usepackage[caption=false]{subfig}
\usepackage[numbers,sort&compress]{natbib}
\bibpunct[, ]{[}{]}{,}{n}{,}{,}
\makeatletter
\def\NAT@def@citea{\def\@citea{\NAT@separator}}
\makeatother

\theoremstyle{plain}
\newtheorem{theorem}{Theorem}[section]

\newtheorem{corollary}[theorem]{Corollary}

\newtheorem{lem}[theorem]{Lemma}
\theoremstyle{definition}

\theoremstyle{remark}
\newtheorem{remark}{Remark}

\renewcommand{\c}{\text{Cos}}
\renewcommand{\S}{\text{Sin}}
\newcommand{\be}{\begin{equation}}
  \newcommand{\ee}{\end{equation}}
  \numberwithin{equation}{section}
\begin{document}
\title{ Indefinite $q$-integrals from a method using $q$-Ricatti equations }
\author{
\name{G. E. Heragy\textsuperscript{a}\thanks{CONTACT G. E. Heragy. Email: moonegam123@gmail.com}, Z. S.I. Mansour, \textsuperscript{b}\thanks{CONTACT Z. S.I. Mansour. Email: zsmansour@cu.edu.eg} and K. M. Oraby\textsuperscript{a}\thanks{CONTACT K. M. Oraby. Email: koraby83@yahoo.com}}
\affil{\textsuperscript{a}Mathematics Department, Faculty of Science, Suez University, Suez, Egypt.
 \\ \textsuperscript{b}Mathematics Department, Faculty of Science, Cairo University, Giza, Egypt.
}
}
\maketitle
\begin{abstract}
Earlier work  introduced a method for obtaining indefinite $q$-integrals of $q$-special functions from the
 second-order linear $q$-difference equations that define them. In this paper, we reformulate the method in terms of $q$-Riccati equations, which are nonlinear and first order. We derive $q$-integrals using fragments of these Riccati equations, and here only
two specific fragment types are examined in detail.
The  results  presented here are
for $q$-Airy function, Ramanujan function, Jackson $q$-Bessel functions, discrete $q$-Hermite polynomials, $q$-Laguerre polynomials, Stieltjes-Wigert polynomial, little $q$-Legendre, and big $q$-Legendre polynomials.
\end{abstract}

\begin{keywords}
$q$-integrals, $q$-Bernoulli fragment, $q$-Linear fragment, Simple algebraic form, $q$-Airy function, Ramanujan function.
\end{keywords}
 \quad \textbf{Mathematics Subject Classification (2020)} 05A30, 33D05, 33D15, 33C10
\section{Introduction and Preliminaries }
In \cite{qqq}, we introduced a  method  to obtain indefinite $q$-integrals of the form
\begin{align}\label{mkl}
 & \int f(x) \Big( \frac{1}{q} D_{q^{-1}} D_qh(x)+   p(x) D_{q^{-1}} h(x) +r(x) h(x)\Big) y(x) d_qx \nonumber\\&=
   f(x/q)\Big (y(x/q) D_{q^{-1}}h(x) - h(x/q) D_{q^{-1}}y(x) \Big)
\end{align}
or
\begin{align}\label{mkrl}
 & \int F(x) \Big( \frac{1}{q} D_{q^{-1}} D_qk(x)+   p(x) D_q k(x) +r(x) k(x)\Big) y(x) d_qx \nonumber\\&=
   F(x)\Big ( y(x) D_{q^{-1}}k(x) - k(x) D_{q^{-1}}y(x) \Big),
\end{align}
where the functions $p(x)$ and $r(x)$ are continuous functions in an interval $I$.
In  $({\ref{mkl}})$ the function $y(x)$ is a solution of
\begin{equation}\label{mxlk}
  \frac{1}{q}D_{q^{-1}} D_q y(x) +  p(x)D_{q^{-1}}y(x) + r(x) y(x) =0,
\end{equation}
$f(x)$ is a solution of
\begin{equation}\label{fht}
  \frac{1}{q} D_{q^{-1}} f(x) =  p(x) f(x)
\end{equation}
and  $ h(x) $  is  an arbitrary function.
In  $({\ref{mkrl}})$ the function  $y(x)$ is a solution of
\begin{equation}\label{mx}
  \frac{1}{q}D_{q^{-1}} D_q y(x) +  p(x)D_q y(x) + r(x) y(x) =0,
\end{equation}
$F(x)$ is a solution of
\begin{equation}\label{frht5}
   D_q F(x) =  p(x) F(x)
\end{equation}
and  $ k(x) $  is an  arbitrary function.
The indefinite $q$-integral
\begin{equation}\label{newmjh}
  \int f(x) d_qx = F(x),
\end{equation}
means that $D_q F(x) = f(x)$, where $D_q$ is Jackson's $q$-difference operator which is  defined in $({\ref{newfgbv}})$ below.
The indefinite $q$-integrals in $({\ref{mkl}})$ and $({\ref{mkrl}})$ generalize Conway's indefinite integral
\begin{equation}\label{intg1}
  \int f(x) \left( \frac{d^2 h}{dx^2} +p(x) \frac{d h}{dx} + r(x) h(x)\right)y(x) dx = f(x) \left( \frac{d h}{dx} y(x) -h(x)  \frac{d y}{dx}\right),
\end{equation}
where $y(x)$ is a solution of
\begin{equation}\label{dif1}
   \frac{d^2 y}{dx^2} +p(x) \frac{d y}{dx} + r(x) y(x)=0,
\end{equation}
 $f(x)$ is a solution of $ f'(x)= p(x) f(x) $ and $ h(x) $  is  an  arbitrary function. See \cite{conway,conwa,conw,con,cpbg,co,c1}.
Conway in \cite{c2,c3}  reformulated $({\ref{intg1}})$  to take  the form
\begin{equation}\label{belown}
  \int f(x) h(x) \left( u'(x) +u^2(x)+ p(x) u(x)  + r(x) \right)y(x)dx = f(x) h(x) \left( u(x) y(x) - y'(x)\right),
\end{equation}
where
\begin{equation*}
  h(x) = \exp (\int u(x) dx),
\end{equation*}
and $u(x)$ is an  arbitrary function.
Then he derived many indefinite integrals by considering fragments of the Ricatti equation \be \nonumber u'(x) +u^2(x)+ p(x) u(x)  + r(x)=0,\ee
of the form
\begin{equation}\label{bern}
  u'(x) +u^2(x)+ p(x) u(x)=0,
\end{equation}
or
\begin{equation}\label{lin}
  u'(x) + p(x) u(x) + r(x)=0.
\end{equation}
He called $({\ref{bern}})$ the Bernoulli fragment, and $({\ref{lin}})$ the linear fragment.\\
This paper is organized as follows. In the remaining of this section, we introduce the  $q$-notations and notions  needed in the sequel. In Section 2, we give a $q$-analogue of Conway's indefinite integral formula in  $({\ref{belown}})$  to the $q$-setting. In Section 3, we introduce indefinite $q$-integrals by considering Bernoulli fragments of the $q$-Ricatti equation. In Section 4, we introduce indefinite $q$-integrals by considering Linear fragments of the $q$-Ricatti equation. Finally,  in Section 5, we introduce new $q$-integrals by setting $u(x) = \frac{a}{x}+b$, with appropriate choice of $a$ and $b$ in $({\ref{4444}})$ and $({\ref{44442nb}})$. Finally, we added an appendix for all $q$-special functions, we used in this paper.
Throughout this paper, $q$ is a positive number less than 1, $\mathbb{N}$ is the set of positive integers,  and $\mathbb{N}_0$ is the set of non-negative integers. We use $I$  to denote an interval with zero or infinity  or an accumulation point.
We follow Gasper and Rahman \cite{rahman} for the definitions of the $q$-shifted factorial,  $q$-gamma, $q$-beta function,  and  $q$-hypergeometric series.\\
A $q$-natural number $[n]_q$ is defined by $ [n]_q=\frac{1-q^n}{1-q},\,\,n\in\mathbb N_0$.
The $q$-derivative $D_qf(x)$ of a function $f$ is defined by
\begin{equation}\label{newfgbv}
  (D_qf)(x)=\frac{f(x)-f(qx)}{(1-q)x}, \text{if} \quad x\neq0,
\end{equation}
and $(D_qf)(0)= f'(0)$ provided that $f'(0)$ exists, see \cite{j, hein}.
Jackson's $q$-integral  of a function $f$   is defined by
\begin{equation}\label{nun}
  \int_0^a f(t)d_qt:=(1-q)a\sum_{n=0}^\infty q^n f(aq^n) ,\,\,a\in\mathbb R,
\end{equation}
provided that the corresponding series in ({\ref{nun}})  converges, see \cite{jackson}.\\
The fundamental theorem of $q$-calculus  \cite[Eq.(1.29)]{Zei}
\begin{equation}\label{fundamen}
  \int_{0}^{a} D_q f(t) d_qt = f(a) -\lim_{n\rightarrow\infty} f(a q^n).
\end{equation}
If $f$ is continuous at zero, then
\begin{equation*}
  \int_{0}^{a} D_q f(t) d_qt = f(a) -f(0).
\end{equation*}
\section{$q$-integrals  from  Ricatti fragments}
In this section, we extend Conway's result $({\ref{belown}})$   to functions  satisfying  homogenous second-order  $q$-difference equation of the form  $({\ref{mxlk}})$ or $({\ref{mx}})$.
\begin{theorem}\label{hnhhbgdds}
Let    $y(x)$ and  $f(x)$  be  solutions  of  Equations $({\ref{mxlk}})$ and $({\ref{fht}})$ in an open interval $I$, respectively.
Let $u(x)$ be a continuous function on $I$ and   $ h(x) $ be  an arbitrary  function  satisfying
\begin{equation}\label{241887}
  D_q h(x)= u(x) h(x) \quad (x \in I).
\end{equation}
 Then
\begin{align}\label{csw}
 &\int f(x) h(x/q) \Big( \frac{1}{q}   D_{q^{-1}} u(x)+\frac{1}{q} u(x) u(x/q)+ A(x) u(x/q)+r(x)\Big) y(x) d_qx \nonumber\\&=  f(x/q) h(x/q) \Big( y(x/q)u(x/q)  -   D_{q^{-1}}y(x)\Big),
\end{align}
where the functions $p(x)$, $r(x)$ are defined  as in ({\ref{mxlk}})  and
\be\label{A1}
 A(x) =  p(x)-\frac{1}{q}x (1-q)r(x).\ee
\end{theorem}
\begin{proof}
Equation ({\ref{mkl}}) can be written as
\begin{align}\label{cw}
 \int f(x) h(x/q) \Big[ \frac{1}{q} \frac{D_{q^{-1}} D_q h(x)}{h(x/q)}+ p(x) \frac{D_{q^{-1}}h(x)}{h(x/q)}+ \frac{r(x) h(x)}{h(x/q)}\Big]y(x) d_qx \nonumber \\  =
    f(x/q) h(x/q) \Big[y(x/q) \frac{D_{q^{-1}}h(x)}{h(x/q)} - D_{q^{-1}}y(x)\Big].
\end{align}
Then from  $({\ref{241887}})$, we get
\begin{align}\label{2147}
 \frac{D_{q^{-1}}D_q h(x)}{h(x/q)}&= D_{q^{-1}} u(x) +\frac{D_{q^{-1}} h(x)}{h(x/q)} u(x)\\&
 = D_{q^{-1}} u(x) + u(x) u(\frac{x}{q}).
 \end{align}
   Also,
 \begin{align}\label{214}
  r(x) \frac{h(x)}{h(x/q)}
 &= \frac{r(x)}{h(x/q)}\left( h(x/q) +(1- \frac{1}{q} ) xD_{q^{-1}} h(x)\right)\\& =
 r(x) \left(1 +(1- \frac{1}{q} ) x u(\frac{x}{q})\right).
\end{align}
Substituting with   $({\ref{2147}})$ and $({\ref{214}})$ into  $({\ref{cw}})$, we get
  $({\ref{csw}})$ and completes the proof.
\end{proof}
 \begin{theorem}\label{kmjjndhyyt}
Let  $y(x)$ and $F(x)$  be  solutions  of  Equations $({\ref{mx}})$ and $({\ref{frht5}})$ in an open interval $I$, respectively.
Let  $u(x)$ be a continuous function on $I$ and  $ k(x) $ be  an arbitrary function satisfying
\begin{equation}\label{21745887s}
  D_{q^{-1}} k(x)= u(x) k(x) \quad x \in I.
\end{equation}  Then
\begin{align}\label{cswq}
 &\int F(x) k(qx) \Big( D_q u(x)+ u(x) u(qx)+ \tilde{A}(x) u(qx)+r(x)\Big)y(x) d_qx = \nonumber\\&  F(x) k(x) \Big( y(x)u(x)  -   D_{q^{-1}}y(x)\Big),
\end{align}
where the functions $p(x)$, $r(x)$ are defined as in ({\ref{mx}}) and
\be\label{A2}
 \tilde{A}(x) =  p(x)+x (1-q)r(x).\ee
\end{theorem}
\begin{proof}
The proof follows similarly as the proof of Theorem ${\ref{hnhhbgdds}}$ and is omitted.\\
\end{proof}
Consider the $q$-Ricatti equation
\begin{equation}\label{yu}
 \frac{1}{q}   D_{q^{-1}} u(x)+\frac{1}{q} u(x) u(x/q)+ A(x) u(x/q)+r(x)= 0,
\end{equation}
and
\begin{equation}\label{yus}
    D_q u(x)+ u(x) u(qx)+ \tilde{A}(x) u(qx)+r(x)= 0,
\end{equation}
 where $ A(x)$ and $ \tilde{A}(x)$ are  defined as in  $({\ref{A1}})$ and  $({\ref{A2}})$, respectively.
  We can prove that Equation $({\ref{yu}})$ ($({\ref{yus}})$) is equivalent to Equation $({\ref{mxlk}})$ ($({\ref{mx}})$) by setting $  \frac{D_q y(x)}{y(x)}= u(x) $ \Big($  \frac{D_{q^{-1}} y(x)}{y(x)}= u(x) $\Big), respectively.
  \vskip 0.5cm The $q$-integrals presented in the sequel are obtained by choosing the  function $ u(x) $ to be a solution of a fragment of the  $q$-Ricatti equations $({\ref{yu}})$ or $({\ref{yus}})$.
  \vskip0.5cm
\begin{theorem}\label{hermith}
Let $n \in \mathbb{N}$ and  $c$ be a real number.
If $ h_n(x;q)$ is the discrete $q$-Hermite I polynomial of degree $n$ which is defined in \eqref{hn1}, then
  \begin{align}\label{her1}
  & \int (q^2x^2;q^2)_\infty \left( (cq+x)[n]_q-x\right) h_n(x;q) d_qx=\nonumber\\&q^{n-1}(1-q) (x^2;q^2)_\infty\left( q h_n(\frac{x}{q};q)- [n]_q(cq+x) h_{n-1}(\frac{x}{q};q)\right),
  \end{align}
  \begin{align}\label{her112}
  & \int x (q^2x^2;q^2)_\infty  h_n(x;q) d_qx= \frac{q^{n-1}(x^2;q^2)_\infty}{[n-1]_q}\left( (1-q) h_n(\frac{x}{q};q)- \frac{1-q^n}{q}x h_{n-1}(\frac{x}{q};q)\right),
  \end{align}
  \begin{align}\label{her2}
    &\int \frac{(q^2x^2;q^2)_\infty}{(q^{-(n+1)}x^2;q^2)_\infty}  h_n(x;q) d_qx= \nonumber\\& \frac{ (x^2;q^2)_\infty}{[n+1]_q(q^{-(n+1)}x^2;q^2)_\infty}\left(  x h_n(\frac{x}{q};q)- q^{n} (1-q^n)  h_{n-1}(\frac{x}{q};q)\right),
  \end{align}
  and
 \begin{align}\label{her3}
    \int x^{n-2}(q^2x^2;q^2)_\infty  h_n(x;q) d_qx= \frac{x^n(x^2;q^2)_\infty}{[n-1]_q}\left(\frac{ h_n(\frac{x}{q};q)}{x}-\frac{1}{q}h_{n-1}(\frac{x}{q};q) \right).
  \end{align}
\end{theorem}
\begin{proof}
The discrete $q$-Hermite I polynomial of degree $n$ is defined in \eqref{hn1} and  satisfies the second order $q$-difference equation \eqref{kjdsmhgb}.
  By comparing   \eqref{kjdsmhgb} with  (${\ref{mxlk}}$), we get
  \begin{equation}\label{needh1}
    p(x)= - \dfrac{x}{1-q}, \quad \quad r(x) =   \frac{ q^{1-n}[n]_q}{1-q}.
  \end{equation}
Then  \be \label{needh2} f(x) =   (q^2x^2;q^2)_\infty \ee is a solution of    (${\ref{fht}}$).
  Therefore Equation $({\ref{csw}})$ becomes
\begin{align}\label{kkk4}
 &\int (q^2x^2;q^2)_\infty h(x/q) \Big( \frac{1}{q} D_{q^{-1}} u(x)+\frac{1}{q} u(x) u(x/q)-\frac{q^{-n}x}{1-q} u(x/q)+ \frac{q^{1-n}[n]_q}{(1-q)}\Big)y(x) d_qx \nonumber\\&=
  (x^2;q^2)_\infty  h(x/q)\Big ( y(x/q)u(x/q)  -   D_{q^{-1}}y(x)\Big).
\end{align}
By taking the fragment
\begin{equation}\label{jnhhygee}
  D_{q^{-1}} u(x) + u(x) u(x/q) =0,
\end{equation}
we get
\be \label{all} u(x) =\frac{1}{x+c}. \ee
Hence \be\label{all2}
h(x) =\left\{\begin{array}{lc}
  1+\frac{x}{c}, & \text{if $c \neq 0$;} \\\\
  x, & \text{ if $c=0$}
\end{array}\right.
 \ee
 is a solution of    $({\ref{241887}})$.
Substituting with $h(x)=1+\frac{x}{c}$ into   $({\ref{kkk4}})$ and using
\begin{equation}\label{4666vgf}
  D_{q^{-1}} h_n(x;q)= [n]_q h_{n-1}(\frac{x}{q};q),
\end{equation} see \cite[Eq.(3.28.7)]{askey}, we get  $({\ref{her1}})$.
Substituting with $h(x)=x$ into   $({\ref{kkk4}})$ and using  $({\ref{4666vgf}})$ we get  $({\ref{her112}})$.
To prove $({\ref{her2}})$, we consider the fragment
\begin{equation*}
  \frac{1}{q} u(x) u(x/q)-\frac{q^{-n}x}{1-q} u(x/q) =0,
\end{equation*}
 then  $ u(x) = \frac{q^{1-n}}{1-q}x$ and
    $ h(x) =\dfrac{1}{( q^{1-n}x^2;q^2)_ \infty}$ is a solution of   $({\ref{241887}})$. Substituting with $h(x)$ and $u(x)$ into   $({\ref{kkk4}})$ and using $({\ref{4666vgf}})$,  we get $({\ref{her2}})$.
Finally, the proof of  $({\ref{her3}})$ follows by taking the fragment
\begin{equation*}
  -\frac{q^{-n}x}{1-q} u(x/q)+ \frac{q^{1-n}[n]_q}{(1-q)}=0.
\end{equation*}
In this case,  $ u(x) = \frac{[n]_q}{x}$ and $ h(x) =  x^n$ is a solution of    $({\ref{241887}})$. Substituting with $h(x)$ and $u(x)$ into   $({\ref{kkk4}})$ and using $({\ref{4666vgf}})$,  we get $({\ref{her3}})$.
\end{proof}
\begin{theorem}\label{59}
Let $n \in \mathbb{N}$ and  $c$ be a real number.
If $\widetilde{h}_n(x;q)$ is the discrete $q$-Hermite II polynomial of degree $n$ which is defined in \eqref{hn2}, then
  \begin{align}\label{her21}
   &\int \frac{1}{(-x^2;q^2)_\infty} \left( qx[n-1]_q+c[n]_q\right)\widetilde{h}_n(x;q) d_qx= \nonumber\\& \frac{1-q}{(-x^2;q^2)_\infty}\left(  \widetilde{h}_n(x;q)- q^{1-n} [n]_q (c+x) \widetilde{h}_{n-1}(x;q)\right),
  \end{align}
   \begin{align}\label{her211}
   &\int \frac{x}{(-x^2;q^2)_\infty} \widetilde{h}_n(x;q) d_qx=  \frac{1-q}{[n-1]_q(-x^2;q^2)_\infty}\left(  \frac{1}{q}\widetilde{h}_n(x;q)- q^{-n} [n]_q x \widetilde{h}_{n-1}(x;q)\right),
  \end{align}
  \begin{align}\label{her22}
  & \int \frac{(-q^{n+3}x^2;q^2)_\infty}{(-x^2;q^2)_\infty}  \widetilde{h}_n(x;q) d_qx= \nonumber\\& \frac{(-q^{n+1}x^2;q^2)_\infty}{[n+1]_q(-x^2;q^2)_\infty}\left( q^n x \widetilde{h}_n(x;q)-  q^{1-n}(1-q^n)\widetilde{h}_{n-1}(x;q)\right),
  \end{align}
  and
  \begin{align}\label{her23}
    &\int \frac{x^{n-2}}{(-x^2;q^2)_\infty}  \widetilde{h}_n(x;q) d_qx= \frac{ x^n}{[n-1]_q(-x^2;q^2)_\infty}\left( \frac{ \widetilde{h}_n(x;q)}{x}-  \widetilde{h}_{n-1}(x;q)\right).
  \end{align}
\end{theorem}
\begin{proof}
The discrete $q$-Hermite II polynomial of degree $n$ is defined in \eqref{hn2} and  satisfies the second order $q$-difference equation \eqref{kjdsmhg}.
  By comparing   \eqref{kjdsmhg} with  (${\ref{mx}}$), we get
  \begin{equation*}
    p(x)= - \dfrac{x}{1-q},\quad \quad r(x) = \frac{ [n]_q}{1-q} .
  \end{equation*}
Then   $ F(x) =   \frac{1}{(-x^2;q^2)_\infty}$  is a solution of  (${\ref{frht5}}$),
  and  $({\ref{cswq}})$ becomes
\begin{align}\label{khki5}
& \int  \frac{k(qx)}{( -x^2;q^2)_\infty} \Big(   D_q u(x)+ u(x) u(qx)-\frac{q^n x}{(1-q)}u(qx)+\frac{[n]_q}{(1-q)}\Big)y(x) d_qx =\nonumber\\&
    \frac{k(x)}{( -x^2;q^2)_\infty}\Big ( y(x)u(x)  -   D_{q^{-1}}y(x)\Big).
\end{align}
Consider the fragment
\begin{equation}\label{jnnyqqse}
  D_q u(x)+ u(x) u(qx) =0.
\end{equation}
 Hence, \be \label{all4} u(x) =\frac{1}{x+c} \ee and
  \begin{equation}\label{all3}
k(x) =\left\{\begin{array}{lc}
  1+\frac{x}{c}, & \text{if $c \neq 0$;} \\\\
  x, & \text{ if $c=0$}
\end{array}\right.
\end{equation}
is a  solution of  $({\ref{21745887s}})$.  Substituting  with $u(x)$ and $k(x)=1+\frac{x}{c} $ into   $({\ref{khki5}})$, and using \cite[Eq.(3.29.7)]{askey}
  (with $x$ is replaced by $\frac{x}{q}$ )
\begin{equation}\label{4666vgfh}
  D_{q^{-1}}\widetilde{h}_n(x;q)=q^{1-n} [n]_q \widetilde{h}_{n-1}(x;q),
\end{equation}  we get  $({\ref{her21}})$.
Substituting  with $u(x)$ and $k(x)=x $ into   $({\ref{khki5}})$, and using $({\ref{4666vgfh}})$ we get  $({\ref{her211}})$.
 Also, by taking the fragment
\begin{equation*}
  u(x) u(qx)-\frac{q^n x}{1-q}u(qx) =0,
\end{equation*}
which has the solution  $u(x) =\frac{q^n }{1-q}x$.  Hence $ k(x) =(-q^{n+1}x^2;q^2)_\infty $  is a solution of  ({\ref{21745887s}}). Substituting  with $k(x)$ into   $({\ref{khki5}})$,  we get   $({\ref{her22}})$.
 Similarly, to prove  $({\ref{her23}})$,   we consider the fragment
 \begin{equation*}
   -\frac{q^n x}{1-q}u(qx)+\frac{[n]_q}{1-q}=0,
 \end{equation*}
  then we obtain $u(x) =q^{1-n}[n]_q\dfrac{1}{x}$ and  $k(x) =x^n$. Substituting  with $u(x)$ and $k(x)$
  into Equation  $({\ref{khki5}})$,  we get   $({\ref{her23}})$.
\end{proof}
\begin{theorem}\label{5h9}
Let $\nu$, $c$ and $q$ be real numbers such that $\nu > 1$.  Let $\alpha \in \mathbb{R}$ be such that $\alpha=\frac{\ln (q^{\nu}+q^{-\nu}-q^{-1})}{\ln q}$. Then
\begin{align}\label{dodovo}
   &\int \dfrac{x}{(- x^2  (1-q)^2;q^2)_\infty} \left(\frac{q^{-1}-q^{-\nu}[\nu]_q^2}{x}-\frac{cq^{-\nu}[\nu]^2}{x^2}+ c+qx \right) J_\nu^{(2)}(  x |q^2) d_qx = \nonumber\\& \frac{x}{q(- x^2  (1-q)^2;q^2)_\infty} \left(  J_\nu^{(2)}(  x |q^2)-(c+x) D_{q^{-1}}  J_\nu^{(2)}( x |q^2) \right),
\end{align}
\begin{align}\label{dodovo1}
   &\int \dfrac{x}{(- x^2  (1-q)^2;q^2)_\infty} \left(\frac{q^{-1}-q^{-\nu}[\nu]_q^2}{x}+qx \right) J_\nu^{(2)}(  x |q^2) d_qx = \nonumber\\& \frac{x}{q(- x^2  (1-q)^2;q^2)_\infty} \left(  J_\nu^{(2)}(  x |q^2)-x D_{q^{-1}}  J_\nu^{(2)}( x |q^2) \right),
\end{align}
  \begin{align}\label{herbhy22}
   &\int \dfrac{x^{\alpha+1}}{(- x^2  (1-q)^2;q^2)_\infty} \left(q+\frac{1-q^{1-\nu}[\nu]^2}{qx^2} \right) J_\nu^{(2)}(  x |q^2) d_qx = \nonumber\\& \frac{x^{\alpha+1}}{q^{\alpha}(- x^2  (1-q)^2;q^2)_\infty} \left( \frac{q^{1-\nu}(1-q)[\nu]^2-1}{x} J_\nu^{(2)}(  x |q^2)- D_{q^{-1}}  J_\nu^{(2)}( x |q^2) \right).
  \end{align}
\end{theorem}
\begin{proof}
  By comparing   \eqref{mnbvxx} with  (${\ref{mx}}$),  we obtain
\begin{equation*}
  p(x) = \frac{1}{x}-q(1-q)x,\quad  r(x) =  q- q^{1-v} [v]^2_q \frac{1}{x^2}.
\end{equation*}
Then $  F(x) = \dfrac{x}{(- x^2  (1-q)^2;q^2)_\infty} $
 is a solution of  (${\ref{frht5}}$),
  and  $({\ref{cswq}})$ becomes
\begin{align}\label{khki5rm}
& \int  \frac{xk(qx)y(x)}{( -x^2(1-q)^2;q^2)_\infty} \Big(   D_q u(x)+ u(x) u(qx)+\frac{1-q^{1-\nu}(1-q)[\nu]^2 }{x}u(qx)+\frac{qx^2-q^{1-\nu}[\nu]^2}{x^2}\Big) d_qx\nonumber\\& =
    \frac{xk(x)}{( -x^2(1-q)^2;q^2)_\infty}\Big ( y(x)u(x)  -   D_{q^{-1}}y(x)\Big).
\end{align}
Considering the fragment  $({\ref{jnnyqqse}})$,   we get $u(x)$ and $k(x)$ as in  $({\ref{all4}})$ and $({\ref{all3}})$, respectively. Substituting  with $u(x)$ and $k(x)=1+\frac{x}{c} $ into  $({\ref{khki5rm}})$
   to obtain  $({\ref{dodovo}})$. Substituting  with $u(x)$ and $k(x)=x$ into  $({\ref{khki5rm}})$
   we obtain  $({\ref{dodovo1}})$.
   To  prove $(\ref{herbhy22})$,
 consider the fragment
\begin{equation*}
  u(x) u(qx)+\frac{1-q^{1-\nu}(1-q)[\nu]^2 }{x}u(qx)=0,
\end{equation*}
which implies that  $u(x) =\frac{q^{1-\nu}+q^{1+\nu}-(1+q)}{(1-q)x}$ and then $ k(x) =x^{\alpha} $  is a solution of   $({\ref{21745887s}})$. Substituting  with $u(x)$ and $k(x)$ into   $({\ref{khki5rm}})$ to obtain $(\ref{herbhy22})$.
\end{proof}
\begin{theorem}\label{kmjuqqqq}
Let  $c$ be a real number.
    If $ Ai_q(x)$ is the $q$-Airy function which is defined in \eqref{ai1}, then
    \begin{align}\label{omnnk}
    & \sum_{n=0}^{\infty} (-1)^n q^n \left(c-1+q^2+q^n x \right)Ai_q(q^n x)\nonumber\\&= \frac{qc+x}{q(1+q)}\,_1\phi_1 (0;-q^2;q,-x)-(1-q)Ai_q(\frac{x}{q}),
    \end{align}
     \begin{align}\label{omnnk1}
    & \sum_{n=0}^{\infty} (-1)^n q^n \left(1-q^2-q^n x \right)Ai_q(q^n x)\nonumber\\&= (1-q)Ai_q(\frac{x}{q}) -\frac{x}{q(1+q)}\,_1\phi_1 (0;-q^2;q,-x),
    \end{align}
    \begin{align}\label{kmjut}
      & \sum_{n=0}^{\infty} q^n  \left(q-q^3-q^n x\right)(-q^{-3}x;q)_n Ai_q(q^n x)\nonumber\\&=
      \frac{x }{1+q}\,_1\phi_1 (0;-q^2;q,-x)-\left(q(1+q)+\frac{x}{q}\right) Ai_q(\frac{x}{q}).
    \end{align}
  \end{theorem}
  \begin{proof}
  The $q$-Airy function is defined in \eqref{ai1} and  satisfies the second order $q$-difference equation \eqref{kjdsmhrfgfggtr}.
     By comparing   \eqref{kjdsmhrfgfggtr} with  (${\ref{mxlk}}$), we get
  \begin{equation}\label{need21}
    p(x)= - \dfrac{1+q}{q(1-q)x}, \quad \quad r(x) =   \frac{ 1}{q(1-q)^2x}.
  \end{equation}
  By taking the fragment  $({\ref{jnhhygee}})$,  we get $u(x)$ and $h(x)$ as in  $({\ref{all}})$ and $({\ref{all2}})$, respectively.
 Therefore  $({\ref{csw}})$  takes the form
  \begin{align}\label{jkiutref}
    &\int f(t) h(t/q) \left( -\frac{q(1+q)+t}{q^2 (1-q)t} u(t/q) +\frac{1}{q(1-q)^2t} \right)y(t) d_qt  \nonumber\\&=f(x/q )h(x/q) \left(y(x/q) u(x/q)-D_{q^{-1}}y(x) \right).
  \end{align}
  Denote the right hand side of Equation  $({\ref{jkiutref}})$ by $H(x)$. I.e
  \begin{equation*}
    H(x)= f(x/q )h(x/q) \left(y(x/q) u(x/q)-D_{q^{-1}}y(x) \right).
  \end{equation*}
  Then from  $({\ref{fundamen}})$, we obtain
  \begin{equation}\label{equ125}
  \int_{0}^{x}  f(t) h(t/q) \left( -\frac{q(1+q)+t}{q^2 (1-q)t} u(t/q) +\frac{1}{q(1-q)^2t} \right)y(t) d_qt = H(x) -\lim _{n \rightarrow \infty} H(q^n x).
  \end{equation}
 From   (${\ref{fht}}$), we obtain
  $f(qx)= (-q) f(x).$ Consequently,
   \begin{equation}\label{hnu}
  f(q^nx)= (-1)^n q^n f(x)\quad (n\in \mathbb{N}_0).
  \end{equation}
   Since
  \begin{equation}\label{hnyyw}
  D_{q^{-1}} Ai_q(x)= \frac{1}{1-q^2}\,_1\phi_1 (0;-q^2;q,-x),
\end{equation}
and using $h(x)=1+\frac{x}{c}$ we get
   \be \label{equ22}H(x)= \frac{f(x)}{c q}\left(\frac{cq+x}{q(1-q^2)}\,_1\phi_1 (0;-q^2;q,-x)-Ai_q(\frac{x}{q}) \right).\ee
 Hence  $\lim _{n \rightarrow \infty} H(q^n x)=0$. From (${\ref{nun}}$),(${\ref{all}}$),  (${\ref{all2}}$) with $c \neq 0$ and (${\ref{equ125}}$), we obtain
  \begin{align}\label{equ12522}
 & \int_{0}^{x}  f(t) h(t/q) \left( -\frac{q(1+q)+t}{q^2 (1-q)t} u(t/q) +\frac{1}{q(1-q)^2t} \right)y(t)\, d_qt \nonumber\\&= \frac{f(x)}{cq (1-q)} \sum_{n=0}^{\infty} (-q)^n \left(c-1+q^2+q^nx\right) y(q^n x) = H(x).
  \end{align}
  Combining equations (${\ref{equ22}})$   and (${\ref{equ12522}}$) yields $({\ref{omnnk}})$.
  Substituting with  $h(x)=x$ we get
   \be \label{equ222}H(x)= \frac{f(x)}{ q}\left(\frac{x}{q(1-q^2)}\,_1\phi_1 (0;-q^2;q,-x)-Ai_q(\frac{x}{q}) \right).\ee
 Hence  $\lim _{n \rightarrow \infty} H(q^n x)=0$. From (${\ref{nun}}$),(${\ref{all}}$),  (${\ref{all2}}$) with $c =0$ and (${\ref{equ125}}$), we obtain
  \begin{align}\label{equ125222}
 & \int_{0}^{x}  f(t) h(t/q) \left( -\frac{q(1+q)+t}{q^2 (1-q)t} u(t/q) +\frac{1}{q(1-q)^2t} \right)y(t)\, d_qt \nonumber\\&= \frac{f(x)}{q (1-q)} \sum_{n=0}^{\infty} (-q)^n \left(-1+q^2+q^nx\right) y(q^n x) = H(x).
  \end{align}
  Combining equations (${\ref{equ222}})$   and (${\ref{equ125222}}$) yields $({\ref{omnnk1}})$.
Similarly, we  prove   $({\ref{kmjut}})$,
  by  taking  the fragment
\begin{equation*}
  \frac{1}{q} u(x) u(x/q)-\frac{q(1+q)+x}{q^2 (1-q)x} u(x/q) =0,
\end{equation*}
 which implies that  $ u(x) = \dfrac{q(1+q)+x}{q(1-q)x}$. Since $h(x)$ satisfies  $({\ref{241887}})$, then
 \begin{equation}\label{kmut}
 h(q^nx)= (-q)^n (\frac{-x}{q^2};q)_n  h(x)\quad (n\in \mathbb{N}_0).
 \end{equation}
Substituting with $u(x)$  into  $({\ref{csw}})$ and using equations $({\ref{hnu}})$, $({\ref{hnyyw}})$ and $({\ref{kmut}})$,  we get   $({\ref{kmjut}})$.
  \end{proof}
\begin{theorem}\label{kmjuqqqqv}
  Let $c \in \mathbb{R}$.  If $ A_q(x)$ is the Ramanujan function which is  defined in \eqref{ram1}, then
    \begin{align}\label{omnnkjn}
     &\sum_{n=0}^{\infty} q^{\frac{n(n-3)}{2}} x^n \left(1-q+qc+q^{n+2}x\right) A_q(q^nx)=
     (1-q) A_q(\frac{x}{q}) - (cq+x) A_q(qx),
    \end{align}
    \begin{align}\label{omnnkjn1}
     &\sum_{n=0}^{\infty} q^{\frac{n(n-3)}{2}} x^n \left(1-q+q^{n+2}x\right) A_q(q^nx)=
     (1-q) A_q(\frac{x}{q}) - x A_q(qx),
    \end{align}
    \begin{align}\label{omnnhk}
    &\sum_{n=0}^{\infty} q^{\frac{n(n-3)}{2}}  \left(1-q^2+q^{n+2}x\right)(x;q)_n A_q(q^nx)=
     \frac{x(1+q)-q}{x} A_q(\frac{x}{q}) - x A_q(qx).
    \end{align}
  \end{theorem}
  \begin{proof}
 The Ramanujan function is defined in \eqref{ram1} and  satisfies the second order $q$-difference equation \eqref{kjdsmhrfgfggnuhtr}.
     By comparing   \eqref{kjdsmhrfgfggnuhtr} with  (${\ref{mxlk}}$), we get
  \begin{equation}\label{need22}
    p(x)=  \dfrac{1-qx}{q(1-q)x^2}, \quad \quad r(x) =  \frac{ 1}{(1-q)^2 x^2}.
  \end{equation}
  By taking the fragment  $({\ref{jnhhygee}})$, we get $u(x)$ and $h(x)$ as in  $({\ref{all}})$ and $({\ref{all2}})$, respectively.
Therefore  $({\ref{csw}})$ takes the form
  \begin{align}\label{jkiutrmef}
    &\int f(t) h(t/q) \left(\frac{1-t(1+q)}{q (1-q)t^2} u(t/q) +\frac{1}{(1-q)^2 t^2} \right)y(t) d_qt \nonumber\\& =f(x/q) h(x/q) \left(y(x/q) u(x/q)- D_{q^{-1}}y(x)\right).
  \end{align}
  Denote the right hand side of Equation  $({\ref{jkiutrmef}})$ by $G(x)$. I.e
  \begin{equation*}
    G(x)= f(x/q )h(x/q) \left(y(x/q) u(x/q)-D_{q^{-1}}y(x) \right).
  \end{equation*}
  Then from $({\ref{fundamen}})$, we get
  \begin{equation}\label{equ125*}
  \int_{0}^{x}  f(t) h(t/q) \left(\frac{1-t(1+q)}{q (1-q)t^2} u(t/q) +\frac{1}{(1-q)^2 t^2} \right)y(t) d_qt = G(x) - \lim _{n \rightarrow \infty} G(q^n x).
  \end{equation}
  From (${\ref{fht}}$), we obtain
  $f(qx)= q^2 x f(x).$ Consequently,
   \begin{equation}\label{hnu1}
  f(q^nx)= q^{\frac{n(n-1)}{2}} x^n f(x)\quad (n\in \mathbb{N}_0).
  \end{equation}
   Since
  \begin{equation}\label{hnyyw1}
  D_{q^{-1}} A_q(x)= \frac{q}{1-q} A_q(qx),
\end{equation}
substituting with $h(x)= 1+\frac{x}{c}$, then
   \be \label{equ22*}G(x)= \frac{f(x)}{cqx}\left( A_q(\frac{x}{q}) - \frac{(cq+x)}{1-q} A_q(qx) \right).\ee

     Hence $\lim _{n \rightarrow \infty} G(q^n x)=0$.  From (${\ref{nun}}$),(${\ref{all}}$),  (${\ref{all2}}$) with $c\neq0$ and (${\ref{equ125*}}$), we obtain
  \begin{align}\label{equ12522*}
 & \int_{0}^{x}  f(t) h(t/q) \left( -\frac{q(1+q)+t}{q^2 (1-q)t} u(t/q) +\frac{1}{q(1-q)^2 t} \right)y(t)\, d_qt \nonumber\\&= \frac{f(x)}{cq (1-q)} \sum_{n=0}^{\infty} q^{\frac{n(n-3)}{2}}x^{n-1} \left(1-q+qc+q^{n+2}x\right) y(q^n x)= G(x).
  \end{align}
  Combining equations (${\ref{equ22*}})$  and   (${\ref{equ12522*}}$) yields $({\ref{omnnkjn}})$.
  Substituting with $h(x)= x$, then
   \be \label{equ222*}G(x)= \frac{f(x)}{qx}\left( A_q(\frac{x}{q}) - \frac{x}{1-q} A_q(qx) \right).\ee

     Hence $\lim _{n \rightarrow \infty} G(q^n x)=0$.  From (${\ref{nun}}$), (${\ref{all}}$), (${\ref{all2}}$) with $c=0$ and (${\ref{equ125*}}$), we obtain
  \begin{align}\label{equ125222*}
 & \int_{0}^{x}  f(t) h(t/q) \left( -\frac{q(1+q)+t}{q^2 (1-q)t} u(t/q) +\frac{1}{q(1-q)^2t} \right)y(t)\, d_qt \nonumber\\&= \frac{f(x)}{q (1-q)} \sum_{n=0}^{\infty} q^{\frac{n(n-3)}{2}}x^{n-1} \left(1-q+q^{n+2}x\right) y(q^n x)= G(x).
  \end{align}
  Combining equations (${\ref{equ222*}})$  and   (${\ref{equ125222*}}$) yields $({\ref{omnnkjn1}})$.
  Similarly, we  prove  $({\ref{omnnhk}})$,
 by taking the fragment
\begin{equation*}
  \frac{1}{q} u(x) u(x/q)+\frac{1-x(1+q)}{q (1-q)x^2} u(x/q) =0,
\end{equation*}
 which implies that  $  u(x) = \frac{x(1+q)-1}{(1-q)x^2}$. Since $h(x)$ satisfies   $({\ref{241887}})$, then
 \begin{equation}\label{kmut1}
 h(q^nx)= \frac{(qx;q)_n}{x^n}  h(x) \quad (n\in \mathbb{N}_0).
 \end{equation}
 Substituting with $u(x)$   into  $({\ref{csw}})$ and using  $({\ref{hnu1}})$, $({\ref{hnyyw1}})$ and $({\ref{kmut1}})$, we get $({\ref{omnnhk}})$.
  \end{proof}
  \begin{corollary}
  Let $ Ai_q(x)$ and $ A_q(x)$ be the $q$-Airy function and the Ramanujan function which are  defined in \eqref{ai1} and  \eqref{ram1}, respectively. Then
    \begin{align*}
    & \sum_{n=0}^{\infty} (-1)^n q^{2n} Ai_q(q^n x)= \frac{q(1-q^2)+x}{qx(1+q)}\,_1\phi_1 (0;-q^2;q,-x)-\frac{(1-q)}{x}Ai_q(\frac{x}{q}),\\&
     \sum_{n=0}^{\infty} q^{\frac{n(n-1)}{2}} x^n  A_q(q^nx) =
     \frac{1-q}{q^2x} A_q(\frac{x}{q}) + \frac{1-q-x}{q^2x} A_q(qx).
    \end{align*}
  \end{corollary}
  \begin{proof}
     Substituting with $c=1-q^2$ and   $c=\frac{-1}{q}(1-q)$    in $({\ref{omnnk}})$ and $({\ref{omnnkjn}})$, respectively, we get the desired results.
  \end{proof}
\vskip0.5cm
 \section {$q$-integrals  from  Bernoulli fragments }
 In this  section, we introduce indefinite $q$-integrals involving
  Bernoulli fragments of $({\ref{yu}})$ or $({\ref{yus}})$,  of the form
\begin{equation}\label{bbbb}
     \frac{1}{q}  D_{q^{-1}} u(x)+\frac{1}{q}u(x) u(x/q)+ C(x) u(x/q)= 0,
   \end{equation}
   or
\begin{equation}\label{bbbb2}
       D_{q} u(x)+u(x) u(qx)+ D(x) u(qx)= 0,
   \end{equation}
   respectively.
    The trivial solution $ u(x) =0 $ of  $({\ref{bbbb}})$  implies that  $ h(x) =c$ is a solution of $({\ref{241887}})$, where $c$ is a non zero constant. Then $({\ref{mkl}})$ becomes
    \begin{equation}\label{d}
      \int f(x)r(x) y(x) d_qx   =  -f(x/q) D_{q^{-1}}y(x).
    \end{equation}
Similarly, the trivial solution $ u(x) =0 $ of $({\ref{bbbb2}})$
 implies that  $ k(x) =c$ is a solution of  $({\ref{21745887s}})$, where $c$ is a non zero constant. Then
    $ ({\ref{mkrl}})$ becomes
    \begin{equation}\label{db}
      \int F(x)r(x) y(x) d_qx   =  -F(x) D_{q^{-1}}y(x).
    \end{equation}

    \begin{theorem}
   Let $n\in\mathbb{N}$. The following statements are true:
     \begin{itemize}
    \item [(a)]
If $h_n(x;q)$ is  the  discrete $q$-Hermite I  polynomial of degree $n$ which is  defined in \eqref{hn1}, then
  \begin{equation}\label{jnido}
  \int (q^2x^2;q^2)_\infty h_n(x;q) d_qx = -q^{n-1}(1-q)(x^2;q^2)_\infty h_{n-1}(\frac{x}{q};q) .
\end{equation}
      \item [(b)] If $ p_n(x;a,b;q)$ is the big $q$-Laguerre  polynomial of degree $n$ which is defined in \eqref{big}, then
  \begin{equation}\label{hnbbgdv}
 \int \dfrac{(\frac{x}{a},\frac{x}{b};q)_\infty }{(x ;q)_\infty}p_n(x;a,b;q) d_qx = \dfrac{abq^2 (1-q) }{(1-aq)(1-bq)}\dfrac{(\frac{x}{aq},\frac{x}{bq};q)_\infty}{(x ;q)_\infty} p_{n-1}(x;aq,bq;q).
\end{equation}
      \item [(c)] If $\alpha > -1$ and  $ L_n^{\alpha}(x;q)$ is the  $q$-Laguerre  polynomial of degree $n$ which is  defined in \eqref{lalpha},  then
  \begin{equation}\label{jmmmmd}
  \int \dfrac{x^\alpha }{(-x;q)_\infty}L_n^{\alpha}(x;q) d_qx = \dfrac{x^{\alpha+1}   }{[n]_q(-x;q)_\infty} L_{n-1}^{\alpha+1}(x;q).
\end{equation}

    \end{itemize}
     \end{theorem}
     \begin{proof}
     The proof of   $(a)$ follows by  substituting with $r(x)$ and $f(x)$ from $({\ref{needh1}})$ and $({\ref{needh2}})$, respectively, into  $({\ref{d}})$.
The proof of $(b)$ follows by comparing   \eqref{kjhgf} with  (${\ref{mxlk}}$) to get
  \begin{equation*}
    p(x)= \dfrac{ x-q(a+b-qab)}{ab q^2  (1-q)(1-x)},\quad \quad r(x) =  - \frac{ q^{-n-1} [n]_q}{ab(1-q)(1-x)} .
  \end{equation*}
Hence $f(x) =  \dfrac{(\frac{x}{a},\frac{x}{b};q)_\infty }{(qx;q)_\infty}$  is a solution of    (${\ref{fht}}$).
  Substituting with $r(x)$ and $f(x)$ into Equation $({\ref{d}})$
  and using
\begin{equation}\label{246}
  D_{q^{-1}}p_n(x;a,b;q)=\dfrac{q^{1-n}[n]_q}{(1-aq)(1-bq)}p_{n-1}(x;aq,bq;q),
\end{equation}
see \cite[Eq.(3.11.7)]{askey},
 we get (${\ref{hnbbgdv}}$).
   To prove  $(c)$,
       compare   \eqref{kjdhg} with  (${\ref{mxlk}}$) to obtain
  \begin{equation*}
    p(x)= \dfrac{ 1-q^{\alpha+1}(1+x)}{q^{\alpha+1}x(1+x)(1-q)},\quad \quad r(x) = \frac{ [n]_q}{x(1-q)(1+x)}.
  \end{equation*}
Hence $ f(x) =  \dfrac{x^{\alpha+1} }{(-qx;q)_\infty}$
is a solution of  ({\ref{fht}}). Finally, we prove (${\ref{jmmmmd}}$)
  by substituting with $r(x)$ and $f(x)$ into  $({\ref{d}})$
   and using
\begin{equation}\label{4587}
  D_{q^{-1}}L_n^{\alpha}(x;q)=\dfrac{- q^{\alpha+1} }{(1-q)}L_{n-1}^{\alpha+1}(x;q),
\end{equation}
see \cite[Eq.(3.21.8)]{askey}.
\end{proof}
\begin{theorem}
The following statements are true:
\begin{itemize}
  \item [(a)] If $\widetilde{h}_n(x;q)$ is the  discrete $q$-Hermite II  polynomial of degree $n$  which is defined in \eqref{hn2}, then
  \begin{equation}\label{jnhhsyl}
  \int \frac{\widetilde{h}_n(x;q)}{(-x^2;q^2)_\infty} d_qx = -\dfrac{q^{1-n}(1-q)}{(-x^2;q^2)_\infty} \widetilde{h}_{n-1}(x;q).
\end{equation}
  \item [(b)]If $\nu$ is a real number,  $ \nu > -1 $, then
  \begin{equation*}
   \int \dfrac{qx^2-q^{1-\nu}[\nu]_q^2}{x(- x^2  (1-q)^2;q^2)_\infty} J_\nu^{(2)}(  x |q^2) d_qx =  \frac{-x}{(- x^2  (1-q)^2;q^2)_\infty} D_{q^{-1}}  J_\nu^{(2)}( x |q^2).
\end{equation*}
\end{itemize}
\end{theorem}
\begin{proof}
The proofs of  $(a)$ and $(b)$  follow by substituting with $r(x)$ and $F(x)$ as in the proof of
 Theorems  ${\ref{59}}$ and  ${\ref{5h9}}$ into Equation $({\ref{db}})$, respectively.
\end{proof}
\begin{theorem}
  The following statements are true:
    \begin{itemize}
      \item  [(a)] If $ Ai_q(x)$ is the $q$-Airy function which is defined in \eqref{ai1}, then
    \begin{equation}\label{hnbmjmbgdv}
    \sum_{k=0}^{\infty} (-q)^k Ai_q(q^kx) = \frac{1}{1+q}\,_1\phi_1 (0;-q^2;q,-x).
    \end{equation}
      \item [(b)] If $ A_q(x)$ is the Ramanujan function which is  defined in \eqref{ram1}, then
    \begin{equation}\label{hnnm}
       \sum_{k=0}^{\infty} q^{\frac{k(k-3)}{2}}x^k A_q(q^k x) = - \,_0\phi_1 (-;0;q,-q^2x).
    \end{equation}
    \item [(c)] If $S_n(x;q)$ is  the  Stieltjes-Wigert  polynomial of degree $n$ $(n \in\mathbb{N})$ which is     defined in \eqref{stief}, then
  \begin{equation}\label{kmjsysy}
  \sum_{k=0}^{\infty} q^{\frac{k(k-3)}{2}}x^k S_n(q^kx;q)=\frac{1}{1-q^n} S_{n-1}(qx;q).
\end{equation}
\end{itemize}
\end{theorem}
\begin{proof}
The proof of  $(a)$ follows by  substituting with $r(x)$ from \eqref{need21} into  $({\ref{d}})$ and using  $({\ref{hnu}})$  and $({\ref{hnyyw}})$.
      The proof of  $(b)$ follows by  substituting with $r(x)$ from   \eqref{need22} into $({\ref{d}})$ and using  $({\ref{hnu1}})$  and \eqref{hnyyw1}.
      To prove $(c)$,  compare  Equation \eqref{kjdshg} with  (${\ref{mxlk}}$) to  get
  \begin{equation*}
    p(x)= \dfrac{ 1-qx}{q x^2 (1-q)},\quad \quad \quad r(x) =  \frac{ [n]_q}{x^2(1-q)} .
  \end{equation*}
  From   (${\ref{fht}}$), we obtain $f(qx)= q^2 x f(x).$ Consequently,
  \begin{equation}\label{kieeee}
   f(q^kx)= q^{\frac{k(k-1)}{2}} x^k f(x) \quad (k \in \mathbb{N}_0).
  \end{equation}
  Substituting  with $ r(x) $  into  $({\ref{d}})$
  and using  (${\ref{nun}}$), (${\ref{kieeee}}$) and
\begin{equation}\label{4666vf}
   D_{q^{-1}}S_n(x;q)=\frac{-q}{1-q}S_{n-1}(qx;q),
\end{equation}
see\cite[Eq.(3.27.7)]{askey},
  we get $({\ref{kmjsysy}})$.
\end{proof}
\vskip0.5cm
\begin{remark}\,
\begin{itemize}
  \item
   The indefinite  $q$-integral (${\ref{hnbbgdv}}$) is nothing else but  \cite[Eq.(42)]{qqq}  or \cite[Eq.(3.11.9)]{askey} (with $n$ is replaced by $n-1$ )
\begin{equation*}
  D_q\left( w(x;aq,bq;q)p_{n-1}(x;aq,bq;q) \right) = \frac{(1-aq)(1-bq)}{abq^2(1-q)}w(x;a,b;q)p_{n}(x;a,b;q),
\end{equation*}
where $ w(x;a,b;q)= \frac{(\frac{x}{b},\frac{x}{a};q)_\infty }{(x;q)_\infty}$ .
   \item
   The indefinite $q$-integral (${\ref{jmmmmd}}$) is equivalent to \cite[Eq.(46)]{qqq}  (if $m=n$ ) and  to \cite[Eq.(3.21.10)]{askey} (if $m=0$) (with $\alpha$ is replaced by $\alpha+1$ and $n$ is replaced by $n-1$ )
\begin{equation*}
  D_q\left( w(x;\alpha+1;q) L_{n-1}^{\alpha+1}(x;q) \right) = [n]_q w(x;\alpha;q)L_{n}^{\alpha}(x;q),
\end{equation*}
where $ w(x;\alpha;q)= \displaystyle{\frac{x^\alpha}{(-x;q)_\infty}}$.
    \item
    The indefinite $q$-integral (${\ref{jnhhsyl}}$) is equivalent to  \cite[Eq.(68)]{qqq} and  to
\begin{equation*}
  D_q\left( w(x;q)\widetilde{h}_{n-1}(x;q) \right) = \frac{q^{n-1}}{1-q} w(x;q)\widetilde{h}_{n}(x;q),
\end{equation*}
where $ w(x;q)= \frac{1}{(-x^2;q^2)}$, see \cite[Eq.(3.29.9)]{askey}.
\end{itemize}
\end{remark}
\vskip0.5cm
We need the following results to prove Theorem  ${\ref{gfdsa}}$.
   \begin{theorem}\label{lem2}
   Let $I$ be an interval containing zero and
  $g(x)$ is a solution of the first order  $q$-difference equation
\begin{equation}\label{ff}
  \frac{1}{q} D_{q^{-1}} g(x) =  A(x) g(x),  \quad g(0)=1,
\end{equation}
where $A(x)$ is the function which is defined in (${\ref{A1}}$).
  Assume that there exists $\gamma$, $0\leq \gamma <1 $ such that $\dfrac{x^\gamma}{g(x)}$ is bounded on $I$.  Then  the function
\begin{equation}\label{lopi}
  v(x)=g(x)\int_{0}^{x} \frac{1}{g(u)} d_qu, \quad x \in I
   \end{equation}
satisfies
\begin{equation}\label{mll}
      D_{q^{-1}}v(x) - q A(x) v(x) =1.
    \end{equation}
   \end{theorem}
   \begin{proof}
Multiplying both sides of Equation $({\ref{mll}})$  by $  \frac{1}{g(x/q)}$ to  obtain
    \begin{equation*}
 \frac{ D_{q^{-1}}v(x)}{g(x/q)}  -\frac{qA(x) v(x)}{g(x/q)}=
 \frac{1}{g(x/q)},
 \end{equation*}
 or equivalently,
 $D_{q}\Big(\frac{v(x)}{g(x)}\Big)=\frac{1}{g(x)} $.
 Hence, from $({\ref{fundamen}})$, we get
 \begin{equation*}
  v(x)=g(x)\int_{0}^{x} \frac{1}{g(u)} d_qu.
   \end{equation*}
 \end{proof}
 \vskip0.5cm
 \begin{theorem}\label{jnuuwwww}
 Assume that $v(x)$ and $g(x)$ are defined as in  Theorem ${\ref{lem2}}$ in an interval $I$ containing zero. Then
   \begin{equation}\label{mnbv}
   \int  \frac{f(x)}{g(x/q)} v(x/q)r(x) y(x)d_qx = \frac{ f(x/q)}{g(x/q)}  \Big( y(x/q)-v(x/q) D_{q^{-1}} y(x)\Big),
   \end{equation}
   where $f(x)$ is a solution of $({\ref{fht}})$ and $y(x)$ is a solution of $({\ref{mxlk}})$.
 \end{theorem}
 \begin{proof}
In $({\ref{241887}})$ set $u(x)=\frac{1}{v(x)}$, where $v(x)$ is defined in Equation $({\ref{lopi}})$. Then
 \begin{equation*}
   \frac{D_qh(x)}{h(x)} = \frac{1}{v(x)} = \frac{\frac{1}{g(x)}}{\int_{0}^{x} \frac{1}{g(u)} d_qu}.
 \end{equation*}
 Hence,
  \begin{equation*}
  D_q \Big (\frac{h(x)}{\int_{0}^{x} \frac{1}{g(u)} d_qu}\Big)=0.
 \end{equation*}
 Therefore,
 \begin{equation*}
   h(x)= c  \int_{0}^{x} \frac{1}{g(u)} d_qu,
 \end{equation*}
 where $c$ is a constant, we can choose $c=1$. Hence,
 \begin{equation*}
  g(x) h(x) = g(x)\int_{0}^{x} \frac{1}{g(u)} d_qu = v(x).
 \end{equation*}
 Substituting with $u(x)$ and $v(x)$ into  Equation $({\ref{csw}})$, we get  Equation $({\ref{mnbv}})$.
 \end{proof}
 \vskip0.5cm
 The  solutions of the second-order $q$-difference equation, see \cite{Zei},
 \begin{equation}\label{gfgf}
       \frac{1}{q}  D_{q^{-1}} D_q y(x) -y(x)=0 \quad (x \in\mathbb{R})
      \end{equation}
      under the initial conditions
      \begin{equation*}
        y(0)=1, \quad D_qy(0)=0, \quad \text{and} \quad y(0)=0, \quad D_qy(0)=1
      \end{equation*}
are the functions $\cos(x;q)$ and $\sin(x;q)$ which are defined in Equation \eqref{coss1} and \eqref{sinn1}, respectively.
\begin{theorem}\label{gfdsa}
\begin{align}\label{cos}
        &\int \dfrac{x \cos(x;q) }{( \frac{-x^2}{q}(1-q)^2;q^2)_\infty}\,_2\phi_1(\frac{-x^2}{q} (1-q)^2,q^2;0;q^2,q)d_qx \nonumber\\&= \frac{-\sqrt{q}}{(\frac{-x^2}{q}(1-q)^2;q^2)_\infty}\left(\dfrac{ \sqrt{q} \cos(\frac{x}{q};q)}{ (1-q)}+ x \sin(q^{\frac{-1}{2}}x;q) \,_2\phi_1(\frac{-x^2}{q} (1-q)^2,q^2;0;q^2,q)\right),
     \end{align}
     and
\begin{align}\label{sin}
      &\int \dfrac{x\sin(x;q) }{ (\frac{-x^2}{q}(1-q)^2;q^2)_\infty}\,_2\phi_1(\frac{-x^2}{q} (1-q)^2,q^2;0;q^2,q)d_qx \nonumber\\& =\frac{1}{(\frac{-x^2}{q}(1-q)^2;q^2)_\infty}\left(  x\cos(q^{\frac{-1}{2}}x;q) \,_2\phi_1(\frac{-x^2}{q}(1-q)^2,q^2;0;q^2,q)-\dfrac{q \sin(x/q;q)}{1-q} \right),
     \end{align}
     where $\sin(x;q)$ and $\cos(x;q)$ are defined in  \eqref{sinn1} and \eqref{coss1}, respectively.
\end{theorem}
\begin{proof}
 By comparing  Equation $({\ref{gfgf}})$  with Equation $({\ref{mxlk}})$, we get $ p(x)=0 $ and $ r(x) = - 1.$ Then  $f(x) =1$ is a solution of $({\ref{fht}})$  and $g(x) = ( -q(1-q)^2 x^2;q^2)_\infty $ is a solution of $({\ref{ff}})$ with
($A(x)= \frac{x}{q}(1-q)$).  By Theorem ${\ref{lem2}}$ $$ v(x) =x(1-q) \,_2\phi_1(- q(1-q)^2 x^2 ,q^2;0;q^2,q). $$ Substituting  with $v(x)$ and $g(x)$ into   $({\ref{mnbv}})$ and using the $q$-difference equation
\eqref{diffff1} and \eqref{difff1}, respectively, we get the desired results.
\end{proof}
\vskip0.5cm
We need the following results to prove Theorem  ${\ref{gfdfsabb}}$ and Theorem  ${\ref{thyj}}$.
\begin{theorem}\label{lem233}
Let $I$ be an interval containing zero and
let $f(x)$ be a solution of the first order $q$-difference equation $({\ref{fht}})$ in $I$.
Assume that there exists $\eta$, $0\leq \eta <1 $ such that $\dfrac{x^\eta}{f(x)}$ is bounded on $I$.
 Then  the function
\begin{equation*}
  v(x)=f(x) \int_{0}^{x} \frac{1}{f(u)} d_qu, \quad x \in I
   \end{equation*}
satisfies
\begin{equation}\label{mljl}
      D_{q^{-1}}v(x) - q p(x) v(x) =1.
    \end{equation}
   \end{theorem}
   \begin{proof}
The proof follows similarly to the proof of Theorem ${\ref{lem2}}$ and is omitted.
 \end{proof}
 \vskip0.5cm
 \begin{theorem}\label{jnuuwwkkww}
  Assume that $v(x)$ and $f(x)$ are defined as in  Theorem ${\ref{lem233}}$ in an interval $I$ containing zero. Then
    \begin{equation}\label{mnbuyv}
   \int  \frac{f(x)}{f(x/q)}\left ( v(x/q)   -\frac{1}{q} x   (1-q) \right) r(x) y(x)d_qx
   =   y(x/q) -  v(x/q) D_{q^{-1}} y(x),
   \end{equation}
   where $y(x)$ is a solution of $({\ref{mxlk}})$.
 \end{theorem}
 \begin{proof}
   The proof follows  similarly to the proof of  Theorem ${\ref{jnuuwwww}}$ and is omitted.
 \end{proof}
 \vskip0.5cm
 \begin{theorem}\label{gfdfsabb}
\begin{align}\label{hnd1}
      & \int x \cos(x;q)d_qx = - q^{-\frac{1}{2}} x \sin( q^{-\frac{1}{2}}x;q) -   \cos(\frac{x}{q} ;q),\\&
       \int x \sin(x;q)d_qx = \frac{x}{q} \cos( q^{-\frac{1}{2}}x;q) -   \sin(\frac{x}{q} ;q),
     \end{align}
     where $\sin(x;q)$ and $\cos(x;q)$ are defined in  \eqref{sinn1} and \eqref{coss1}, respectively.
\end{theorem}
\begin{proof}
From   $({\ref{gfgf}})$, we have $ p(x)=0 $, $ r(x) = - 1  $  and  $f(x) =1$.
By Theorem ${\ref{lem233}}$, we get $v(x) = x $
is a solution of Equation $({\ref{mljl}})$. Substituting  with $v(x)$, $f(x)$ and using the $q$-difference equations
\eqref{diffff1} and \eqref{difff1}, respectively, we get the desired results.
\end{proof}
\begin{theorem}\label{thyj}
Let $n\in \mathbb{N}$.  If $p_n(x;-1;q)$ is the big $q$-Legendre polynomial which is  defined in \eqref{bigg1}, then
 \begin{align}\label{lmmmuyfv}
  &\int \frac{x p_n(x;-1;q)}{(q^2-x^2)} \left(\,_2\phi_1 ( \frac{x^2}{q^2} , q^{2};x^2;q^2,q)-1 \right)d_qx = \nonumber\\&   (q^{-n}-1)[n+1]_q \left(  p_n(\frac{x}{q};-1;q)+ \frac{x}{q} \,_2\phi_1 ( \frac{x^2}{q^2} , q^{2};x^2;q^2,q)\,_3\phi_2 ( q^{-n} , q^{n+1},x;q;-q;q,q) \right).
 \end{align}
\end{theorem}
\begin{proof}
By comparing   \eqref{kjhgdd} with  (${\ref{mxlk}}$),  we get
\begin{equation*}
 p(x) = \dfrac{ x(1+q)}{q^2  (x^2-1)},\quad  r(x) =   - \frac{[n]_q [n+1]_q }{q^{1+n}(x^2-1)}.
\end{equation*}
Then $f(x) = (1-x^2)$ is a solution of  (${\ref{fht}}$).
From Theorem ${\ref{lem233}}$, the function
\begin{equation*}
  v(x) = x(1-q) \,_2\phi_1 ( x^2 , q^{2};x^2 q^2;q^2,q)
\end{equation*}
is a solution of   $({\ref{mljl}})$. Substituting  with $v(x)$ and  $f(x)$  into   $({\ref{mnbuyv}})$ and using
\begin{equation}\label{need}
 D_{q^{-1}}  p_n(x;-1;q) = \frac{-1}{1-q} \,_3\phi_2 ( q^{-n} , q^{n+1},x;q;-q;q,q),
\end{equation}
 we obtain $({\ref{lmmmuyfv}})$.
\end{proof}
\vskip0.5cm
\section {$q$-integrals  from  linear fragments}
 In this  section, we introduce indefinite $q$-integrals involving
 linear fragment of  $({\ref{yu}})$ or $({\ref{yus}})$,  of the form
\begin{equation}\label{bbcc}
     \frac{1}{q}  D_{q^{-1}} u(x)+ p(x) u(x/q)+r(x)= 0,
   \end{equation}
 or
\begin{equation}\label{bbcc2}
       D_{q} u(x)+ p(x) u(qx)+r(x)= 0,
   \end{equation}
   respectively.
\begin{lem}\label{jnyyeffff}
Let $I$ be an interval containing zero.
 Let $p(x)$ and  $r(x)$  be continuous functions at zero. If $f(x)$ is a solution of Equation $({\ref{fht}})$,
then
   \begin{equation}\label{bvv}
   u(x) = \frac{-1}{f(x)}\int_{0}^{qx} f(t) r(t) d_qt
    \end{equation}
is a solution of  Equation $({\ref{bbcc}})$ in $I$.
\end{lem}
\begin{proof}
Multiplying both sides of  $({\ref{bbcc}})$ by $  f(x)$,  we obtain
    \begin{equation*}
   D_{q^{-1}}\Big( f(x) u(x)\Big) = - qf(x)r(x),
 \end{equation*}
 or equivalently $$ D_{q}\Big( f(x) u(x)\Big) = - qf(qx)r(qx).$$ Hence, from $({\ref{fundamen}})$, we get  $({\ref{bvv}})$ and completes the proof.
 \end{proof}
 \vskip0.5cm
If $u(x)$ is a solution of the $q$-linear fragment $({\ref{bbcc}})$, then Equation $({\ref{csw}})$ becomes
 \begin{align}\label{mnxxv}
  & \int f(x) h(x/q) \left(\frac{1}{q} u(x) u(x/q) +\frac{1}{q} x r(x) (q-1)u(x/q)\right)y(x) d_qx \nonumber \\& = f(x/q) h(x/q) \left(y(x/q) u(x/q)- D_{q^{-1}} y(x)\right).
      \end{align}
\begin{theorem}\label{gfdfsa}
\begin{align}\label{mkiu1}
      & \int \frac{x^2  }{(x^2 q^{-1}  (1-q);q^2)_\infty}\cos(x;q)d_qx  = \frac{q }{(\frac{x^2}{q} (1-q);q^2)_\infty}\Big(x \cos(\frac{x}{q} ;q) +\sqrt{q}\sin( q^{\frac{-1}{2}}x;q)\Big),
     \end{align}
\begin{align}\label{mkiu2}
      & \int \frac{x^2   }{ (x^2 q^{-1} (1-q);q^2)_\infty}\sin(x;q)d_qx  = \frac{q}{(\frac{x^2}{q} (1-q);q^2)_\infty} \Big(x \sin( \frac{x}{q};q) - \cos (q^\frac{-1}{2}x;q)\Big),
     \end{align}
     where $\sin(x;q)$ and $\cos(x;q)$ are defined in Equation \eqref{sinn1} and \eqref{coss1}, respectively.
\end{theorem}
\begin{proof}
From Equation  $({\ref{gfgf}})$, we have $ p(x)=0 $, $ r(x) = - 1  $  and  $f(x) =1$ is a solution of Equation (${\ref{fht}}$).
By Lemma ${\ref{jnyyeffff}}$, we get $ u(x) =qx $
satisfies Equation $({\ref{bbcc}})$. Hence \be h(x) = \frac{1 }{(q x^2 (1-q);q^2)_\infty}\ee is a solution of Equation $({\ref{241887}})$. By substituting  with $u(x)$ and  $h(x)$ into Equation $({\ref{mnxxv}})$
and using the $q$-difference equation
\eqref{diffff1} and \eqref{difff1}, respectively, we get the desired results.
\end{proof}
\begin{theorem}
Let $n\in\mathbb{N}$. If $ p_n(x|q)$ is the little $q$-Legendre  polynomials defined in \eqref{little1},   $r_n=\frac {2-q^{-n}-q^{n+1}}{1-q}$, then
 \begin{align}\label{jmnhygfrr}
& \int \frac{x(qx;q)_\infty }{(q r_n x;q)_\infty}p_n(x|q) d_qx \nonumber\\&= \frac{q^{n} x(x;q)_\infty}{[n]_q [n+1]_q(r_nx;q)_\infty}\left(\frac{1}{q}(1-x)\,_2\phi_1(q^{-n+1} , q^{n+2};q^2;q,x)- p_n(\frac{x}{q}|q)\right).
  \end{align}
\end{theorem}
\begin{proof}
By comparing  Equation \eqref{kjhgd} with  (${\ref{mxlk}}$),  we get
\begin{equation*}
  p(x) = \frac{qx+x-1}{qx(qx-1)},\quad r(x) =   \frac{ [n]_q [n+1]_q}{q^nx (1-qx)} .
\end{equation*}
Then $f(x) = x(1-qx)$ is a solution of  (${\ref{fht}}$).
 From  $({\ref{bvv}})$, we get  $u(x) = \frac{-q^{1-n}[n]_q[n+1]_q}{(1-qx)}$, $h(x)$ satisfies the $q$-difference equation $({\ref{241887}})$. Consequently, $ h(x)=\frac{(qx;q)_\infty}{(r_nqx;q)_\infty}$.
By substituting  with $u(x)$ and $h(x)$ into  $({\ref{mnxxv}})$
and using the $q$-difference equation
\begin{equation}\label{4666}
  D_{q^{-1}}p_n(x|q)=-q^{-n} [n]_q [n+1]_q  \,_2\phi_1(q^{1-n} , q^{n+2};q^2;q,x),
\end{equation}
 we get  $({\ref{jmnhygfrr}})$.
\end{proof}
\begin{theorem}
 Let $n\in\mathbb{N}$. If $p_n(x;-1;q)$ is the big $q$-Legendre polynomials defined in \eqref{bigg1}, $r_n= \frac{2-q^{-n}-q^{n+1}}{1-q}$, then
 \begin{align}\label{knywwww}
  &\int \frac{x^2 (x^2;q^2)_\infty}{(r_n x^2;q^2)_\infty} p_n(x;-1;q) d_qx=\nonumber\\&
  \frac{q^{n+2}(\frac{x^2}{q^2};q^2)_\infty}{[n]_q[n+1]_q(\frac{r_nx^2}{q^2};q^2)_\infty}\left(
  \frac{q^{n+1}(1-\frac{x^2}{q^2})}{[n+1]_q(1-q^n)}\,_3\phi_2 ( q^{-n} , q^{n+1},x;q;-q;q,q)-x p_n(\frac{x}{q};-1;q)\right).
 \end{align}
\end{theorem}
\begin{proof}
From  Equation \eqref{kjhgdd}, we have
\begin{equation*}
 p(x) = \dfrac{ x(1+q)}{q^2  (x^2-1)},\quad  r(x) =   - \frac{[n]_q [n+1]_q }{q^{1+n}(x^2-1)}.
\end{equation*}
Then $f(x) = (1-x^2)$ is a solution of (${\ref{fht}}$).
 From  $({\ref{bvv}})$, we have
\begin{equation*}
  u(x) =-q^{-n} [n]_q [n+1]_q \frac{x}{1-x^2},
\end{equation*}
and $ h(x) =\frac{(x^2;q^2)_\infty}{(r_n x^2;q^2)_\infty}$ is a solution of  $({\ref{241887}})$.
By substituting with $u(x)$ and $h(x)$ into  $({\ref{mnxxv}})$
and using the $q$-difference equation $({\ref{need}})$,
we get  $({\ref{knywwww}})$.
\end{proof}
\section {$q$-integrals from substitution of simple algebraic forms}
In this section, we  substitute into Equation ({\ref{csw}}) with simple algebraic forms for $u(x)$ which involve arbitrary constants,
such as
\begin{equation}\label{nhnn}
  u(x)=\frac{a}{x}+b,
\end{equation}
to derive indefinite $q$-integrals. Set
\begin{equation}\label{4444}
  S_q(x) :=\frac{1}{q}  D_{q^{-1}} u(x)+\frac{1}{q}u(x) u(x/q)+ A(x) u(x/q)+r(x).
\end{equation}
Then  ({\ref{csw}}) will be
\begin{equation}\label{567}
  \int f(x) h(x/q) S_q(x) y(x) d_qx = f(x/q) h(x/q) \left(y(x/q) u(x/q) - D_{q^{-1}} y(x)\right),
\end{equation}
where the constants  $a$ and $b$ in Equation ({\ref{nhnn}}) are chosen so that $S_q(x)$ has a simple
form.
Also, we  define
\begin{equation}\label{44442nb}
  T_q(x) := D_{q} u(x)+u(x) u(qx)+ \tilde{A}(x) u(qx)+r(x).
\end{equation}
Then  ({\ref{cswq}}) will be
\begin{equation}\label{567nb}
  \int F(x) k(qx) T_q(x) y(x) d_qx = F(x) k(x) \left( y(x) u(x) - D_{q^{-1}} y(x)\right).
\end{equation}
\begin{theorem}
Let $n \in \mathbb{N}$, $n\geq 2$.
   Let $h_n(x;q)$ be  the  discrete $q$-Hermite I  polynomial of degree $n$ which is defined in \eqref{hn1}.  Then
   \begin{align}\label{2145fdd}
 & \int x (q^2x^2;q^2)_\infty h_n(x;q) d_qx=
 \frac{(1-q)x(x^2;q^2)_\infty}{[n]_q-1} \left(\frac{q^n}{x}h_n(\frac{x}{q};q)-q^{n-1}[n]_q h_{n-1}(\frac{x}{q};q)  \right),
   \end{align}
   and
   \begin{align}\label{nbdgss}
     & \int x^{n-2} (q^2x^2;q^2)_\infty h_n(x;q) d_qx=
\frac{ x^{n}(x^2;q^2)_\infty}{[n-1]_q} \left(\frac{h_n(\frac{x}{q};q)}{x}- \frac{1}{q}h_{n-1}(\frac{x}{q};q)  \right).
   \end{align}
\end{theorem}
\begin{proof}
From  \eqref{kjdsmhgb},
\begin{equation*}
    p(x)= - \dfrac{x}{1-q}, \quad \quad r(x) =   \frac{ q^{1-n}[n]_q}{1-q}.
  \end{equation*}
  Hence $f(x) =   (q^2x^2;q^2)_\infty $ is a solution of   Equation (${\ref{fht}}$).
Set $u(x)$ as in (${\ref{nhnn}}$).
Then
\begin{equation}\label{njuffffj1}
S_q(x) =\frac{a(a-1)}{x^2}+\frac{ab(1+q)}{qx}+\frac{q^{1-n}([n]_q-a)-q^{-n}bx}{1-q} +\frac{b^2}{q}.
\end{equation}
If ${\small a=1}$ and $b=0$ in $({\ref{njuffffj1}})$,  then    \begin{equation*}
                                       S_q(x)= \frac{q^{2-n}[n-1]_q}{1-q},
                                     \end{equation*}
and $h(x) = x$ is a solution of  $({\ref{241887}})$. By  substituting with $u(x)$, $S_q(x)$ and $h(x)$ into  ({\ref{567}}) and using  $({\ref{4666vgf}})$, we get  $({\ref{2145fdd}})$.
 If  $a=[n]_q$ and $b=0$ in $({\ref{njuffffj1}})$, then   \begin{equation*}  S_q(x)= q[n]_q [n-1]_q \frac{1}{x^2}.\end{equation*} Hence  $h(x) = x^n$ is a solution of  $({\ref{241887}})$. Substituting with $u(x)$, $S_q(x)$ and $h(x)$ into  ({\ref{567}}) and using  $({\ref{4666vgf}})$, we get  $({\ref{nbdgss}})$.
\end{proof}
\begin{theorem}
  Let $n \in \mathbb{N}$, $n\geq 2$.  Let $\widetilde{h}_n(x;q)$ be the  discrete $q$-Hermite II  polynomial of degree $n$ which is defined in \eqref{hn2}. Then
   \begin{align}\label{hbggdmkl}
     &\int \frac{x}{(-x^2;q^2)_\infty}\widetilde{h}_n(x;q) d_qx = \dfrac{(1-q)x}{[n-1]_q(-x^2;q^2)_\infty} \left( \frac{\widetilde{h}_n(x;q)}{qx}-q^{-n} [n]_q \widetilde{h}_{n-1}(x;q)\right),
   \end{align}
   and
   \begin{align}\label{hbggdmklgn}
    &\int \frac{x^{n-2}}{(-x^2;q^2)_\infty} \widetilde{h}_n(x;q) d_qx = \dfrac{x^n}{[n-1]_q(-x^2;q^2)_\infty} \left( \frac{\widetilde{h}_n(x;q)}{x}- \widetilde{h}_{n-1}(x;q)\right).
   \end{align}
\end{theorem}
\begin{proof}
From  \eqref{kjdsmhg},
\begin{equation*}
    p(x)= - \dfrac{x}{1-q}, \quad \quad r(x) =   \frac{ [n]_q}{1-q}.
  \end{equation*}
  Then   $ F(x) =   \frac{1}{(-x^2;q^2)_\infty}$  is a solution of  (${\ref{frht5}}$).
Set $u(x)$ as in   (${\ref{nhnn}}$), we get
\begin{equation}\label{mdifhfhfh}
T_q(x) =\frac{a(a-1)}{qx^2}+\frac{ab(1+q)}{qx}+\frac{[n]_q-q^n(q^{-1}a+bx)}{1-q} +b^2.
\end{equation}
If $a=1$ and $b=0$ in $({\ref{mdifhfhfh}})$,  then  \begin{equation*}
                           T_q(x)= \frac{[n-1]_q}{1-q}.
                          \end{equation*}
Therefore $k(x) = x$ is a solution of  $({\ref{21745887s}})$.  By  substituting with $u(x)$, $T_q(x)$ and $k(x)$ into $({\ref{567nb}})$  and using Equation $({\ref{4666vgfh}})$,  we get  $({\ref{hbggdmkl}})$. If $a=q^{1-n} [n]_q$ and $b=0$ in $({\ref{mdifhfhfh}})$, then
\begin{equation*}
  T_q(x)= q^{1-2n}[n]_q[n-1]_q \dfrac{1}{x^2}.
\end{equation*}
 Therefore $k(x) = x^n$ is a solution of  $({\ref{21745887s}})$.  By  substituting with $u(x)$, $T_q(x)$ and $k(x)$ into  $({\ref{567nb}})$  and using  $({\ref{4666vgfh}})$,  we get $({\ref{hbggdmklgn}})$.
\end{proof}
\vskip0.5cm
\begin{theorem}
 Let $n \in \mathbb{N}$. If $S_n(x;q)$ is  the  Stieltjes-Wigert  polynomial of degree $n$ defined in \eqref{stief},   then
   \begin{align}\label{1}
   \sum_{k=0}^{\infty} q^{\frac{k(k-5)}{2}+nk} x^k S_n(q^kx;q) = S_n(\frac{x}{q};q)+\frac{x}{1-q^n} S_{n-1}(qx;q),
  \end{align}
  and
  \begin{align}\label{12112}
   \sum_{k=0}^{\infty} q^{\frac{k(k-3)}{2}} x^k \left(1+q^{2+k}[n-1]_qx\right)S_n(q^kx;q) = S_n(\frac{x}{q};q)+\frac{x}{1-q} S_{n-1}(qx;q).
  \end{align}
  \end{theorem}
  \begin{proof}
 By comparing   \eqref{kjdshg} with  (${\ref{mxlk}}$), we get
  \begin{equation*}
    p(x)= \dfrac{ 1-qx}{q x^2 (1-q)},\quad \quad \quad r(x) =  \frac{ [n]_q}{x^2(1-q)}.
  \end{equation*}
Set $u(x)$ as in (${\ref{nhnn}}$).
Then
\begin{align}\label{njuffffj}
S_q(x) &=\frac{b+q[n]_q}{q(1-q)x^2}+\frac{a(a-[n]_q)}{x^2}+\frac{ab(1+q)-qb[n-1]_q}{qx} \nonumber\\&+\frac{a(1-x)}{(1-q)x^3}-\frac{b}{q(1-q)x}+\frac{b^2}{q}.
\end{align}
We set $a = [n]_q $ and $b=0$ in $({\ref{njuffffj}})$ then $$ S_q(x)= \frac{[n]_q}{(1-q)x^3}.$$ Therefore, $h(x) = x^n$ is a solution of   $({\ref{241887}})$. By substituting with $h(x)$, $S_q(x)$ and $u(x)$ into  $({\ref{567}})$, using
 $({\ref{hnu1}})$  and $({\ref{4666vf}})$,
we get  $({\ref{1}})$.\\ If $a = 1 $ and $b=0$ in $({\ref{njuffffj}})$,  then $$ S_q(x)= \frac{1+q^2[n-1]_q x}{(1-q)x^3}.$$ Therefore, $h(x) = x$ ia a solution of    $({\ref{241887}})$. By substituting with $h(x)$, $S_q(x)$ and $u(x)$ into $({\ref{567}})$ and using
 $({\ref{hnu1}})$  and $({\ref{4666vf}})$,
we get  $({\ref{12112}})$.
\end{proof}
\vskip0.5cm
\begin{remark}
\begin{itemize}\,
  \item The indefinite $q$-integral ({\ref{nbdgss}}) is  equivalent to  \cite[Eq.(67)]{qqq} ( if $m=n$ ) and to ({\ref{her3}}) in Theorem {\ref{hermith}}.\\
  \item The indefinite $q$-integral ({\ref{hbggdmklgn}}) is  equivalent to   \cite[Eq.(69)]{qqq} (if $m=n$ ) and    to   ({\ref{her23}}) in Theorem {\ref{59}}.
\end{itemize}
\end{remark}
\vskip0.5cm
\begin{center}
  \bf{Appendix }
\end{center}
\setcounter{equation}{0}
\renewcommand{\theequation}{I.\arabic{equation}}
\small{
The big $q$-Laguerre  polynomial \be \label{big} p_n(x;a,b;q):= \,_3\phi_2\left( \begin{array}{cccc}
                     q^{-n} , 0,x \\
                    aq,bq
                   \end{array}
\mid q;q\right) \ee  satisfies the second-order $q$-difference equation, see \cite[Eq.(3.11.5)]{askey},
\begin{equation}\label{kjhgf}
    \frac{1}{q} D_{q^{-1}} D_qy(x)+  \dfrac{ x-q(a+b-qab)}{ab q^2  (1-q)(1-x)}D_{q^{-1}}y(x) - \frac{ q^{-n-1} [n]_q}{ab(1-q)(1-x)} y(x) =0.
\end{equation}
The  $q$-Laguerre  polynomial of degree $n$
\begin{align}\label{lalpha}
  L_n^{\alpha}(x;q):= \frac{1}{(q;q)_n}\,_2\phi_1\left( \begin{array}{cccc}
                     q^{-n} ,-x \\
                    0
                   \end{array}
\mid q;q^{n+\alpha+1}\right),  \quad  \quad  \alpha > -1, n \in \mathbb{N}
\end{align}
   satisfies the second-order $q$-difference equation, see \cite[Eq.(3.21.6)]{askey},
\begin{align}\label{kjdhg}
    \frac{1}{q} D_{q^{-1}} D_qy(x)+  \dfrac{ 1-q^{\alpha+1}(1+x)}{q^{\alpha+1}x(1+x)(1-q)}D_{q^{-1}}y(x) + \frac{ [n]_q}{x(1-q)(1+x)} y(x) =0.
\end{align}
The  Stieltjes-Wigert  polynomials \be \label{stief}S_n(x;q):=\frac{1}{(q;q)_n}\,_1\phi_1\left( \begin{array}{cccc}
                     q^{-n}  \\
                    0
                   \end{array}
\mid q;-q^{n+1}x\right), \quad (n \in \mathbb{N}_0) \ee  satisfies the second-order $q$-difference equation, see \cite[Eq.(3.27.5)]{askey},
\begin{align}\label{kjdshg}
    \frac{1}{q} D_{q^{-1}} D_qy(x)+  \dfrac{ 1-qx}{q x^2 (1-q)}D_{q^{-1}}y(x) + \frac{ [n]_q}{x^2(1-q)} y(x) =0.
\end{align}
The discrete $q$-Hermite I polynomial of degree $n$  \be\label{hn1} h_n(x;q):= q^{\binom{n}{2}} \,_2\phi_1\left( \begin{array}{cccc}
                     q^{-n},x^{-1}  \\
                    0
                   \end{array}
\mid q;-q x\right), n\in \mathbb{N}_0 \ee   satisfies the second-order $q$-difference equation, see \cite[Eq.(3.28.5)]{askey},
\begin{align}\label{kjdsmhgb}
    \frac{1}{q} D_{q^{-1}} D_qy(x)- \dfrac{x}{1-q}D_{q^{-1}}y(x) + \frac{ q^{1-n}[n]_q}{1-q} y(x) =0.
\end{align}
The discrete $q$-Hermite II polynomials of degree $n$  \be \label{hn2}\widetilde{h}_n(x;q):= x^n \,_2\phi_1\left( \begin{array}{cccc}
                     q^{-n},q^{-n+1} \\
                    0
                   \end{array}
\mid q^2;\frac{-q^2}{x^2}\right), n\in \mathbb{N}_0 \ee  satisfies the second-order $q$-difference equation, see \cite[Eq.(3.29.5)]{askey},
\begin{align}\label{kjdsmhg}
    \frac{1}{q} D_{q^{-1}} D_qy(x)- \dfrac{x}{1-q}D_q y(x) + \frac{ [n]_q}{1-q} y(x) =0.
\end{align}
The second Jackson $q$-Bessel function  \be \label{jac1} J_\nu^{(2)}(  x |q^2):=J_{\nu}^{(2)}(2 x(1-q);q^2)\ee satisfies the second-order $q$-difference equation \cite{ra}
\begin{align}\label{mnbvxx}
&\frac{1}{q} D_{q^{-1}}D_q y(x) + \dfrac{1-q  x^2 (1-q)}{x} D_{q} y(x) +\dfrac{q   x^2  - q^{1-v} [v]^2_q }{x^2} y(x)=0.
\end{align}
The  $q$-Airy function \be \label{ai1}Ai_q(x):= \,_1\phi_1 (0;-q;q,-x)\ee satisfies the second-order $q$-difference Equation, see \cite[Eq.(4)]{airy},
\begin{align}\label{kjdsmhrfgfggtr}
    \frac{1}{q} D_{q^{-1}} D_{q}y(x)- \dfrac{1+q}{qx(1-q)}D_{q^{-1}} y(x) +\frac{ 1}{qx(1-q)^2} y(x) =0.
\end{align}
The  Ramanujan function \be \label{ram1}A_q(x):= \,_0\phi_1 (-;0;q,-qx)\ee satisfies the second-order $q$-difference Equation, see \cite[Eq.(5)]{airy},
\begin{align}\label{kjdsmhrfgfggnuhtr}
    \frac{1}{q} D_{q^{-1}} D_{q}y(x)+ \dfrac{1-qx}{qx^2(1-q)}D_{q^{-1}} y(x) +\frac{ 1}{x^2(1-q)^2} y(x) =0.
\end{align}
 The big $q$-Legendre polynomials  \be \label{bigg1}p_n(x;-1;q) := \,_3\phi_2\left( \begin{array}{cccc}
                     q^{-n} , q^{n+1},x \\
                    q,-q
                   \end{array}
\mid q;q\right)  \ee  satisfies the second-order $q$-difference equation, see \cite[Eq.(3.5.17)]{askey},
\begin{align}\label{kjhgdd}
    \frac{1}{q} D_{q^{-1}} D_qy(x)+  \dfrac{ x(1+q)}{q^2  (x^2-1)}D_{q^{-1}}y(x) - \frac{[n]_q [n+1]_q}{q^{1+n}(x^2-1)} y(x) =0.
\end{align}
The little $q$-Legendre  polynomials \be \label{little1}p_n(x|q) := \,_2\phi_1\left( \begin{array}{cccc}
                     q^{-n} , q^{n+1} \\
                    q
                   \end{array}
\mid q;qx\right)  \ee   satisfies the second-order $q$-difference equation, see \cite[Eq.(3.12.16)]{askey},
\begin{align}\label{kjhgd}
    \frac{1}{q} D_{q^{-1}} D_qy(x)+   \frac{qx+x-1}{qx(qx-1)} D_{q^{-1}}y(x) + \frac{ [n]_q [n+1]_q}{q^nx (1-qx)}y(x) =0.
\end{align}
Jackson also introduced   three $q$-analogs of Bessel functions, \cite{jackson,rahman}, they  are defined by
\begin{align*}
  J_{\nu}^{(1)}(z;q):= \frac{(q^{v+1};q)_\infty}{(q;q)_\infty} \sum_{n=0}^{\infty} \dfrac{(-1)^n }{(q,q^{v+1};q)_n} (z/2)^{2n+\nu},\quad |z|<2,
\end{align*}
\begin{align*}
  J_{\nu}^{(2)}(z;q):= \frac{(q^{v+1};q)_\infty}{(q;q)_\infty} \sum_{n=0}^{\infty} \dfrac{(-1)^n q^{n(n+\nu)} }{(q,q^{v+1};q)_n} (z/2)^{2n+\nu},\quad z\in \mathbb C,
\end{align*}
\begin{align*}
  J_{\nu}^{(3)}(z;q):= \frac{(q^{v+1};q)_\infty}{(q;q)_\infty} \sum_{n=0}^{\infty} \dfrac{(-1)^n q^{\frac{n(n+1)}{2}} }{(q,q^{v+1};q)_n} (z)^{2n+\nu},\quad z\in \mathbb C.
\end{align*}
There are three known $q$-analogs of the trigonometric  functions, \{$\sin_q z$,$\cos_q z$\}, \{$\S_q z$,$\c_q z$\} and \{$\sin(z;q)$,$\cos(z;q)$\}. Each set of $q$-analogs is related to one of the three $q$-analogs of Bessel functions.\\
The functions $\sin_q z$ and $\cos_q z$ are defined for $|z| < \frac{1}{1-q}$ by
\begin{align*}
\sin_q z &:=\frac{(q^2;q^2)_\infty}{(q;q^2)_\infty} (z)^{1/2} J^{(1)}_{1/2} (2 z;q^2) = \sum_{n=0}^{\infty} (-1)^n  \frac{z^{2n+1}}{[2n+1]_q!},
\end{align*}
\begin{align*}
\cos_q z &:=\frac{(q^2;q^2)_\infty}{(q;q^2)_\infty} (z)^{1/2} J^{(1)}_{-1/2} (2 z;q^2)  = \sum_{n=0}^{\infty} (-1)^n  \frac{z^{2n}}{[2n]_q!}.
\end{align*}
The functions $\S_q z$ and $\c_q z$ are defined for $ z \in \mathbb C$ by
\begin{align*}
\S_q z&:=\frac{(q^2;q^2)_\infty}{(q;q^2)_\infty} (z)^{1/2} J^{(2)}_{1/2} (2 z;q^2)  = \sum_{n=0}^{\infty} (-1)^n  \frac{q^{2n^2+n} z^{2n+1}}{[2n+1]_q!},
\end{align*}
\begin{align*}
\c_q z&:=\frac{(q^2;q^2)_\infty}{(q;q^2)_\infty} (z)^{1/2} J^{(2)}_{-1/2} (2 z;q^2)  = \sum_{n=0}^{\infty} (-1)^n  \frac{q^{2n^2-n} z^{2n}}{[2n]_q!}.
\end{align*}
Finally, the functions $\sin(z;q)$ and $\cos(z;q)$ are defined for $ z \in \mathbb C$ by
\begin{align}\label{sinn1}
\sin(z;q)&:=\frac{(q;q)_\infty}{(q^{1/2};q)_\infty} z^{1/2} J^{(3)}_{1/2} (z(1-q);q^2) = \sum_{n=0}^{\infty} (-1)^n \frac{q^{n^2+n} z^{2n+1}}{\Gamma_q (2n+2)},
\end{align}
\begin{align}\label{coss1}
\cos(z;q)&:=\frac{(q;q)_\infty}{(q^{1/2};q)_\infty} (z q ^{-1/2})^{1/2} J^{(3)}_{-1/2} (z(1-q)/\sqrt{q};q^2)  = \sum_{n=0}^{\infty} (-1)^n  \frac{q^{n^2} z^{2n}}{\Gamma_q (2n+1)}.
\end{align}
The $q$-trigonometric  functions satisfy the $q$-difference equations
\begin{align}\label{inhb2555}
  D_{q^{-1}} \sin_q z =  \cos_q (\frac{z}{q}), \quad \quad D_{q^{-1}}  \cos_q z = - \sin_q (\frac{z}{q}).
\end{align}
\begin{align}\label{pssa}
  D_{q^{-1}} \S_q z =  \c_q (z), \quad \quad D_{q^{-1}} \c_q z = - \S_q (z).
\end{align}
\begin{align}\label{diffff1}
  D_{q^{-1}} \sin( z;q) =  \cos(q^{\frac{-1}{2}}  z;q).
\end{align}
\begin{align}\label{difff1}
  D_{q^{-1}} \cos( z;q) = -q^{\frac{1}{2}} \sin(q^{\frac{-1}{2}} z;q).
\end{align}
}

\end{document}